\documentclass[12pt]{article}
\usepackage{amssymb}
\usepackage{amsthm}
\usepackage{amsmath}
\usepackage{graphicx}
\usepackage{enumerate}
\usepackage{pict2e}
\usepackage{framed}

\DeclareMathOperator{\Con}{Con}

\newtheorem{theorem}{Theorem}[section]
\newtheorem{definition}[theorem]{Definition}
\newtheorem{lemma}[theorem]{Lemma}
\newtheorem{proposition}[theorem]{Proposition}

\newtheorem{remark}[theorem]{Remark}
\newtheorem{example}[theorem]{Example}

\title{Implications in pseudocomplemented and Stone lattices}
\author{Ivan~Chajda and Helmut~L\"anger}
\date{}

\begin{document}

\footnotetext{Support of the research of the first author by IGA, project P\v rF~2024~011, and of the second author by the Austrian Science Fund (FWF), project I~4579-N, entitled ``The many facets of orthomodularity'', is gratefully acknowledged.}

\maketitle

\begin{abstract}
Two kinds of the connective implication are introduced as term operations of a pseudocomplemented lattice. It is shown that they share a lot of properties with the intuitionistic implication based on Heyting algebras. In particular, if the pseudocomplemented lattice in question is a Stone lattice then the considered implications satisfy some kind of quasi-commutativity, of the exchange property, some version of adjointness with the meet-operation and some kind of the derivation rule Modus Ponens and of the contraposition law. Two kinds of deductive systems are defined and their elementary properties are shown. All investigated concepts are illuminated by examples.
\end{abstract}



\section{Introduction}

Starting with L.~E.~J.~Brouwer \cite{B08}, \cite{B13} and A.~Heyting \cite H, intuitionistic logic is formalized by so-called Heyting algebras, i.e.\ bounded distributive relatively pseudocomplemented lattices where the connective implication $\rightarrow$ is realized by the relative pseudocomplement $*$ where $a*b$, the {\em relative pseudocomplement} of $a$ with respect to $b$, is the greatest element $x$ satisfying $a\wedge x\le b$. This means that the binary operations $\wedge$ and $*$ are connected via the {\em adjointness property}
\[
a\wedge x\le b\text{ if and only if }x\le a*b.
\]
In such a case the negation $a^*$ of $a$ is defined by $a^*:=a*0$. The element $a^*$ is called the {\em pseudocomplement} of $a$.

The question arises if some kind of implication can be derived by means of pseudocomplementation. This means that if $\mathbf L=(L,\vee,\wedge,{}^*,0,1)$ is a bounded distributive lattice with pseudocomplementation and an implication is introduced as a term of $\mathbf L$, may such an implication have similar nice properties as in the case of relative pseudocomplementation? For example, does it satisfy the derivation rule Modus Ponens or the contraposition law? The aim of the present paper is to show how such an implication can be introduced and which restrictions are necessary in order to obtain a formalization of an implication similar to that considered in intuitionistic logic.

\section{Preliminaries}

In the whole paper we work with a bounded lattice $\mathbf L=(L,\vee,\wedge,0,1)$ which is {\em pseudocomplemented}, i.e.\ for each $a\in L$ there exists a greatest element $a^*\in L$ satisfying $a\wedge a^*=0$. This means that for all $x\in L$ we have
\[
a\wedge x=0\text{ if and only if }x\le a^*.
\]
This element $a^*$ is called the {\em pseudocomplement} of $a$. Such lattices were introduced by G.~Gr\"atzer and E.~T.~Schmidt in \cite{GS}. It is almost evident that every finite distributive lattice is pseudocomplemented. A distributive pseudocomplemented lattice is called a {\em Brouwerian lattice}. However, there are also non-distributive pseudocomplemented lattices, see Example~\ref{ex1} below. If $\mathbf L$ is pseudocomplemented then it will be denoted by $(L,\vee,\wedge,{}^*,0,1)$. Let us recall several familiarly known properties of pseudocomplementation.

\begin{lemma}\label{lem3}
	Let $(L,\vee,\wedge,{}^*,0,1)$ be a pseudocomplemented lattice and $a,b\in L$. Then the following holds:
\begin{enumerate}[{\rm(i)}]
	\item $0^*=1$, $1^*=0$, $a\le a^{**}$, $a^{***}=a^*$,
	\item $a\le b$ implies $b^*\le a^*$,
	\item $(a\vee b)^*=a^*\wedge b^*$,
	\item $(a\wedge b)^{**}=a^{**}\wedge b^{**}$,
	\item $a\wedge(a^*\wedge b)^*=a$.
\end{enumerate}
{\rm(}See e.g.\ {\rm\cite{BH}}, {\rm\cite{CHK}}, {\rm\cite F} and {\rm\cite G.)}
\end{lemma}

For the convenience of the reader we provide a proof.

\begin{proof}[Proof of Lemma~\ref{lem3}]
\
\begin{enumerate}[(i)]
\item and (ii) are easy.
\item[(iii)] We have $(a\vee b)^*\le a^*,b^*$ and hence $(a\vee b)^*\le a^*\wedge b^*$. On the other hand $a\le a^{**}\le(a^*\wedge b^*)^*$ and similarly $b\le(a^*\wedge b^*)^*$ which implies $a\vee b\le(a^*\wedge b^*)^*$ and therefore $a^*\wedge b^*\le(a^*\wedge b^*)^{**}\le(a\vee b)^*$.
\item[(iv)] Since $(a\wedge b)^{**}\le a^{**},b^{**}$ we have $(a\wedge b)^{**}\le a^{**}\wedge b^{**}$. Now the following are equivalent:
\begin{align*}
   	          a\wedge b\wedge(a\wedge b)^* & =0, \\
	                  b\wedge(a\wedge b)^* & \le a^*, \\
	     a^{**}\wedge b\wedge(a\wedge b)^* & =0, \\
	             a^{**}\wedge(a\wedge b)^* & \le b^*, \\
	a^{**}\wedge b^{**}\wedge(a\wedge b)^* & =0, \\
	                   a^{**}\wedge b^{**} & \le(a\wedge b)^{**}.
\end{align*}
\item[(v)] We have $a\le a^{**}\le(a^*\wedge b)^*$.
\end{enumerate}
\end{proof}

Further, a distributive pseudocomplemented lattice is called a {\em Stone lattice} if it satisfies the {\em Stone identity}
\[
x^*\vee x^{**}\approx1.
\]

\begin{lemma}\label{lem6}
	Let $(L,\vee,\wedge,{}^*,0,1)$ be a Stone lattice and $a,b\in L$. Then the following holds:
	\begin{enumerate}[{\rm(i)}]
	\item $(a\vee b)^{**}=a^{**}\vee b^{**}$,
	\item $(a\wedge b)^*=a^*\vee b^*$.
	\end{enumerate}
\end{lemma}

For the convenience of the reader we provide a proof. First we need an additional lemma.

\begin{lemma}\label{lem7}
	Let $(L,\vee,\wedge,0,1)$ be a bounded distributive lattice and $a,b,c,d\in L$. Then the following holds:
	\begin{enumerate}[{\rm(i)}]
		\item If $a$ and $b$ are complemented to each other then $a$ and $b$ are pseudocomplements of each other.
		\item If $a$ and $c$ are complemented to each other and $b$ and $d$ are complemented to each other then $a\vee b$ and $c\wedge d$ are complemented to each other and $a\wedge b$ and $c\vee b$ are complemented to each other.
	\end{enumerate}
\end{lemma}

\begin{proof}[Proof of Lemma~\ref{lem7}]
	\
	\begin{enumerate}[(i)]
	\item If $a\wedge c=0$ then because of $a\vee b=1$ we have $c=c\wedge(a\vee b)=(c\wedge a)\vee(c\wedge b)=c\wedge b\le b$. If, conversely, $c\le b$ then $a\wedge c\le a\wedge b=0$.
	\item We have
	\begin{align*}
 		  (a\vee b)\vee(c\wedge d) & =(a\vee b\vee c)\wedge(b\vee c\vee d)=1\wedge1=1, \\
		(a\vee b)\wedge(c\wedge d) & =(a\wedge c\wedge d)\vee(b\wedge c\wedge d)=0\vee0=0.
	\end{align*}
	The second assertion can be proved in a similar way.
	\end{enumerate}
\end{proof}

\begin{proof}[Proof of Lemma~\ref{lem6}]
	\
	\begin{enumerate}
		\item[(ii)] Using Lemma~\ref{lem7} we see that since $a^*$ and $a^{**}$ are complemented to each other and $b^*$ and $b^{**}$ are complemented to each other, also $a^*\vee b^*$ and $a^{**}\wedge b^{**}$ are complemented to each other and hence $a^*\vee b^*=(a^{**}\wedge b^{**})^*$ which implies $(a^*\vee b^*)^{**}=a^*\vee b^*$. Using (i), (iv) and (iii) of Lemma~\ref{lem3} we conclude
		\[
		(a\wedge b)^*=(a\wedge b)^{***}=(a^{**}\wedge b^{**})^*=(a^*\vee b^*)^{**}=a^*\vee b^*.
		\]
		\item[(i)] Using (iii) of Lemma~\ref{lem3} and (ii) we obtain $(a\vee b)^{**}=(a^*\wedge b^*)^*=a^{**}\vee b^{**}$.
	\end{enumerate}
\end{proof}

Hence Stone lattices satisfy the identities $(x\vee y)^*\approx x^*\wedge y^*$ and $(x\wedge y)^*\approx x^*\vee y^*$.

An {\em element} $a\in L$ is called {\em dense} if $a^*=0$. Let $D(\mathbf L)$ denote the set of all dense elements of $\mathbf L$.

\begin{example}\label{ex1}
	Consider the pseudocomplemented lattices visualized in Figure~1:
	
\vspace*{-3mm}

\begin{center}
	\setlength{\unitlength}{7mm}
	\begin{picture}(22,8)
		\put(2,1){\circle*{.3}}
		\put(1,4){\circle*{.3}}
		\put(3,3){\circle*{.3}}
		\put(3,5){\circle*{.3}}
		\put(2,7){\circle*{.3}}
		\put(2,1){\line(-1,3)1}
		\put(2,1){\line(1,2)1}
		\put(3,3){\line(0,1)2}
		\put(2,7){\line(-1,-3)1}
		\put(2,7){\line(1,-2)1}
		\put(1.85,.3){$0$}
		\put(.35,3.85){$b$}
		\put(3.4,2.85){$a$}
		\put(3.4,4.85){$c$}
		\put(1.85,7.4){$1$}
		\put(1.6,-.75){{\rm(a)}}
		\put(9,1){\circle*{.3}}
		\put(7,3){\circle*{.3}}
		\put(11,3){\circle*{.3}}
		\put(9,5){\circle*{.3}}
		\put(13,5){\circle*{.3}}
		\put(11,7){\circle*{.3}}
		\put(9,1){\line(-1,1)2}
		\put(9,1){\line(1,1)4}
		\put(11,3){\line(-1,1)2}
		\put(11,7){\line(-1,-1)4}
		\put(11,7){\line(1,-1)2}
		\put(8.85,.3){$0$}
		\put(6.35,2.85){$a$}
		\put(11.4,2.85){$b$}
		\put(8.35,4.85){$c$}
		\put(13.4,4.85){$d$}
		\put(10.85,7.4){$1$}
		\put(8.6,-.75){{\rm(b)}}
		\put(19,1){\circle*{.3}}
		\put(17,3){\circle*{.3}}
		\put(21,3){\circle*{.3}}
		\put(19,5){\circle*{.3}}
		\put(19,7){\circle*{.3}}
		\put(19,1){\line(-1,1)2}
		\put(19,1){\line(1,1)2}
		\put(19,5){\line(-1,-1)2}
		\put(19,5){\line(1,-1)2}
		\put(19,5){\line(0,1)2}
		\put(18.85,.3){$0$}
		\put(16.35,2.85){$a$}
		\put(21.4,2.85){$b$}
		\put(19.4,4.85){$c$}
		\put(18.85,7.4){$1$}
		\put(18.6,-.75){{\rm(c)}}
		\put(8.2,-1.75){Fig.\ 1}
		\put(5.6,-2.75){Pseudocomplemented lattices}
	\end{picture}
\end{center}

\vspace*{20mm}

For these three lattices we have
\[
\begin{array}{l|lllll}
	x            & 0 & a & b & c & 1 \\
	\hline
	x^*          & 1 & b & c & b & 0 \\
	\hline
	x^{**}       & 0 & c & b & c & 1 \\
\end{array}
\quad
\begin{array}{l|llllll}
	x            & 0 & a & b & c & d & 1 \\
	\hline
	x^*          & 1 & d & a & 0 & a & 0 \\
	\hline
	x^{**}       & 0 & a & d & 1 & d & 1
\end{array}
\quad
\begin{array}{l|llllll}
	x            & 0 & a & b & c & 1 \\
	\hline
	x^*          & 1 & b & a & 0 & 0 \\
	\hline
	x^{**}       & 0 & a & b & 1 & 1 \\
\end{array}
\]
\hspace*{17mm} $D(\mathbf L)=\{1\}$\hspace*{16mm} $D(\mathbf L)=\{c,1\}$\hspace*{19mm} $D(\mathbf L)=\{c,1\}$

\vspace*{3mm}

\hspace*{32mm} {\rm(a)}\hspace*{38mm} {\rm(b)}\hspace*{36mm} {\rm(c)}

\vspace*{4mm}

The lattices in {\rm(b)} and {\rm(c)} are distributive whereas the lattice in {\rm(a)} is not. The lattice in {\rm(a)} is not a Stone lattice though it satisfies the Stone identity since it is not distributive, the lattice in {\rm(b)} is a Stone lattice, but the lattice in {\rm(c)} is not a Stone lattice since $a^*\vee a^{**}=b\vee a=c\ne1$.
\end{example}

\begin{lemma}
	Let $\mathbf L=(L,\vee,\wedge,{}^*,0,1)$ be a distributive pseudocomplemented lattice and $a\in L$. Then $a\vee a^*\in D(\mathbf L)$.
\end{lemma}

\begin{proof}
According to Lemma~\ref{lem3} we have $(a\vee a^*)^*=a^*\wedge a^{**}=0$.
\end{proof}

\begin{remark}
	It is worth noticing that the Stone identity does not imply distributivity. Namely, the lattice in Fig.~1 {\rm(a)} satisfies the Stone identity, but it is not distributive.
\end{remark}

\section{An implication in pseudocomplemented lattices}

On a pseudocomplemented lattice $(L,\vee,\wedge,{}^*,0,1)$ we define a binary term operation $\rightarrow$ by
\[
x\rightarrow y:=x^*\vee y
\]
for all $x,y\in L$. Obviously,
\begin{itemize}
	\item $x^{**}\rightarrow y\approx x\rightarrow y\approx x\rightarrow(x\rightarrow y)$,
	\item $y\le x\rightarrow y$ for all $x,y\in L$,
	\item $x\rightarrow0\approx x^*$.
\end{itemize}

\begin{example}
	The operation tables of $\rightarrow$ for the lattices from Figure~1 are as follows:
	\[
	\begin{array}{l|lllll}
		\rightarrow & 0 & a & b & c & 1 \\
		\hline
		0           & 1 & 1 & 1 & 1 & 1 \\
		a           & b & 1 & b & 1 & 1 \\
		b           & c & c & 1 & c & 1 \\
		c           & b & 1 & b & 1 & 1 \\
		1           & 0 & a & b & c & 1
	\end{array}
	\quad
	\begin{array}{l|llllll}
		\rightarrow & 0 & a & b & c & d & 1 \\
		\hline
		0           & 1 & 1 & 1 & 1 & 1 & 1 \\
		a           & d & 1 & d & 1 & d & 1 \\
		b           & a & a & c & c & 1 & 1 \\
		c           & 0 & a & b & c & d & 1 \\
		d           & a & a & c & c & 1 & 1 \\
		1           & 0 & a & b & c & d & 1
	\end{array}
	\quad
	\begin{array}{l|lllll}
		\rightarrow & 0 & a & b & c & 1 \\
		\hline
		0           & 1 & 1 & 1 & 1 & 1 \\
		a           & b & c & b & c & 1 \\
		b           & a & a & c & c & 1 \\
		c           & 0 & a & b & c & 1 \\
		1           & 0 & a & b & c & 1
	\end{array}
	\]
\hspace*{34mm}{\rm(a)}\hspace*{38mm}{\rm(b)}\hspace*{37mm}{\rm(c)}
\end{example}

\begin{remark}
	If $(L,\vee,\wedge,{}^*,0,1)$ is a pseudocomplemented lattice and $a\in L$ then one can easily check the following:
\[
\begin{array}{l|lllll}
	\rightarrow & 0      & a         & a^*            & a^{**}         & 1 \\
	\hline
	0           & 1      & 1         & 1              & 1              & 1 \\
	a           & a^*    & a\vee a^* & a^*            & a^*\vee a^{**} & 1 \\
	a^*         & a^{**} & a^{**}    & a^*\vee a^{**} & a^{**}         & 1 \\
	a^{**}      & a^*    & a\vee a^* & a^*            & a^*\vee a^{**} & 1 \\
	1           & 0      & a         & a^*            & a^{**}         & 1
\end{array} 
\]
\end{remark}

We are going to list several elementary properties of this implication including the so-called exchange property.

\begin{lemma}\label{lem4}
	Let $\mathbf L=(L,\vee,\wedge,{}^*,0,1)$ be a pseudocomplemented lattice and $a,b,c\in L$. Then the following holds:
	\begin{enumerate}[{\rm(i)}]
		\item If $a\le b$ then $a\rightarrow b\in D(\mathbf L)$,
		\item $(a\rightarrow b)\vee c=a\rightarrow(b\vee c)=(a\rightarrow c)\vee b$,
		\item $a\rightarrow(b\rightarrow c)=(a\rightarrow c)\vee(b\rightarrow c)=b\rightarrow(a\rightarrow c)$ {\rm(}{\em exchange property}{\rm)},
		\item $a\rightarrow(b\rightarrow a)\in D(\mathbf L)$,
		\item if $a\le b$ then $c\rightarrow a\le c\rightarrow b$ and $b\rightarrow c\le a\rightarrow c$.
	\end{enumerate}
\end{lemma}

\begin{proof}
	\
	\begin{enumerate}[(i)]
		\item If $a\le b$ then $(a\rightarrow b)^*=(a^*\vee b)^*\le a^{**},b^*$ and hence $(a\rightarrow b)^*\le a^{**},a^*$ which implies $(a\rightarrow b)^*\le a^{**}\wedge a^*=0$, i.e.\ $a\rightarrow b\in D$.
		\item We have $(a\rightarrow b)\vee c=a^*\vee b\vee c=a\rightarrow(b\vee c)=a\rightarrow(c\vee b)=(a\rightarrow c)\vee b$.
		\item We have
		\begin{align*}
			a\rightarrow(b\rightarrow c) & =a^*\vee(b^*\vee c)=(a^*\vee c)\vee(b^*\vee c)=(a\rightarrow c)\vee(b\rightarrow c)= \\
			                             & =(b\rightarrow c)\vee(a\rightarrow c)=b\rightarrow(a\rightarrow c).
		\end{align*}
		\item Since $a\le b\rightarrow a$, (iv) follows from (i).
		\item If $a\le b$ then $c\rightarrow a=c^*\vee a\le c^*\vee b=c\rightarrow b$ and $b\rightarrow c=b^*\vee c\le a^*\vee c=a\rightarrow c$.
	\end{enumerate}
\end{proof}

If $\mathbf L$ is, moreover, distributive, we can prove even the following.

\begin{lemma}\label{lem1}
	Let $\mathbf L=(L,\vee,\wedge,{}^*,0,1)$ be a distributive pseudocomplemented lattice and $a,b,c\in L$. Then the following holds:
	\begin{enumerate}[{\rm(i)}]
		\item $a\wedge(a\rightarrow b)\le b$,
		\item $a\le b\rightarrow c$ implies $a\wedge b\le c$,
		\item if $a\rightarrow b=1$ then $a\le b$.
	\end{enumerate}
\end{lemma}

\begin{proof}
\
\begin{enumerate}[(i)]
	\item $a\wedge(a\rightarrow b)=a\wedge(a^*\vee b)=(a\wedge a^*)\vee(a\wedge b)\le b$.
	\item If $a\le b\rightarrow c$ then $a\wedge b\le(b\rightarrow c)\wedge b=(b^*\vee c)\wedge b=(b^*\wedge b)\vee(c\wedge b)\le c$.
	\item If $a\rightarrow b=1$ then $a^*\vee b=1$ thus, by distributivity,
	\[
	a\le a\vee b=(a\vee b)\wedge1=(a\vee b)\wedge(a^*\vee b)=(a\wedge a^*)\vee b=b.
	\]
\end{enumerate}
\end{proof}

Let us note that condition (i) of Lemma~\ref{lem1} can be read as follows: The truth value of proposition $y$ cannot be less than the truth value of the conjunction of the propositions $x$ and $x\rightarrow y$, which is nothing else than the Modus Ponens derivation rule.

Observe that assertion (ii) of Lemma~\ref{lem1} does not hold if $\mathbf L$ is not distributive. For example, consider the lattice $N_5$ depicted in Fig.~1 (a). Then we have $c\le1=c\rightarrow a$, but $c\wedge c=c\not\le a$.

In a Stone lattice, we can prove a modified version of adjointness between $\wedge$ and $\rightarrow$, see the following lemma.

\begin{proposition}\label{prop1}
	Let $(L,\vee,\wedge,{}^*,0,1)$ be a Stone lattice and $a,b,c\in L$. Then the following holds:
	\begin{enumerate}[{\rm(i)}]
		\item $(a^*\rightarrow b^*)\rightarrow b^*=a^*\vee b^*=(b^*\rightarrow a^*)\rightarrow a^*$,
		\item $a\wedge b^*\le c$ if and only if $a\le b^*\rightarrow c$.
	\end{enumerate}
\end{proposition}

\begin{proof}
	\
	\begin{enumerate}[(i)]
		\item According to Lemma~\ref{lem3} we have
\begin{align*}
(a^*\rightarrow b^*)\rightarrow b^* & =(a^{**}\vee b^*)^*\vee b^*=(a^*\wedge b^{**})\vee b^*=(a^*\vee b^*)\wedge(b^{**}\vee b^*)= \\
                                    & =a^*\vee b^*=b^*\vee a^*=(b^*\rightarrow a^*)\rightarrow a^*.
\end{align*}
		\item If $a\wedge b^*\le c$ then
	\[
	a\le(b^{**}\vee a)\wedge(b^{**}\vee b^*)=b^{**}\vee(a\wedge b^*)\le b^{**}\vee c=b^*\rightarrow c.
	\]
	The converse implication follows from Lemma~\ref{lem1}.
	\end{enumerate}
\end{proof}

We can characterize Stone lattices by using properties of the implication $\rightarrow$.

\begin{theorem}
	Let $\mathbf L_1=(L,\vee,\wedge,\rightarrow,0,1)$ be an algebra of type $(2,2,2,0,0)$ such that $(L,\vee,\wedge,0,1)$ is a bounded distributive lattice and put $x^*:=x\rightarrow0$ for all $x\in L$. Then $\mathbf L_2:=(L,\vee,\wedge,{}^*,0,1)$ is a Stone lattice if and only if $\mathbf L_1$ satisfies the identities
	\begin{enumerate}[{\rm(1)}]
		\item $x\wedge(0\rightarrow0)\approx x$,
		\item $x\wedge(x\rightarrow0)\approx0$,
		\item $x\wedge\big((x\wedge y)\rightarrow0\big)\approx x\wedge(y\rightarrow0)$,
		\item $(x\rightarrow0)\vee\big((x\rightarrow0)\rightarrow0\big)\approx1$.
	\end{enumerate}
\end{theorem}

\begin{proof}
	Let $a,b\in L$. If $\mathbf L_2$ is a Stone lattice then
	\begin{enumerate}[(1)]
		\item $x\wedge(0\rightarrow0)\approx x\wedge0^*\approx x\wedge1\approx x$,
		\item $x\wedge(x\rightarrow0)\approx x\wedge x^*\approx0$,
		\item Using (ii) of Lemma~\ref{lem6} we obtain $x\wedge\big((x\wedge y)\rightarrow0\big)\approx x\wedge(x\wedge y)^*\approx x\wedge(x^*\vee y^*)\approx(x\wedge x^*)\vee(x\wedge y^*)\approx x\wedge y^*\approx x\wedge(y\rightarrow0)$,
		\item $(x\rightarrow0)\vee\big((x\rightarrow0)\rightarrow0\big)\approx x^*\vee x^{**}\approx1$,
	\end{enumerate}
	i.e.\ $\mathbf L_1$ satisfies identities (1) -- (4). Conversely, assume $\mathbf L_1$ to satisfy identities (1) -- (4). If $a\wedge b=0$ the according to (1) and (3) we have
	\[
	b=b\wedge(0\rightarrow0)=b\wedge\big((b\wedge a)\rightarrow0\big)=b\wedge(a\rightarrow0)\le a\rightarrow0=a^*.
	\]
	If, conversely, $b\le a^*$ then $b\le a\rightarrow0$ and according to (2) we have $a\wedge b\le a\wedge(a\rightarrow0)=0$. This shows that $a^*$ is the pseudocomplement of $a$. Finally, (4) is nothing else than the Stone identity. Altogether, we obtain that $\mathbf L_2$ is a Stone lattice.
\end{proof}

The previous theorem shows that the class of Stone lattices can be described by identities using the operation symbols $\vee$, $\wedge$, $\rightarrow$ and $0$. Hence, the class of these implication algebras forms a variety similar to that of intuitionistic logic based on Heyting algebras.

\section{Another possible implication in pseudocomplemented lattices}

We can introduce a connective implication in a pseudocomplemented lattice $(L,\vee,\wedge,{}^*,0,1)$ also in another way as follows:
\[
x\Rightarrow y:=x^*\vee y^{**}
\]
for all $x,y\in L$. Obviously,
\begin{itemize}
	\item $x\Rightarrow y\approx x\rightarrow y^{**}$,
	\item $x\Rightarrow y^*\approx x\rightarrow y^*$,
	\item $x\Rightarrow y\approx x^{**}\Rightarrow y\approx x\Rightarrow y^{**}\approx x^{**}\Rightarrow y^{**}$,
	\item $x\rightarrow y\le x\Rightarrow y$ for all $x,y\in L$,
	\item $x\Rightarrow y\approx y^*\Rightarrow x^*\approx y^*\rightarrow x^*$ ({\em contraposition law}),
	\item $x\Rightarrow0\approx x^*$.
\end{itemize}

\begin{example}
	The operation tables of $\Rightarrow$ for the lattices from Figure~1 are as follows:
	\[
	\begin{array}{l|lllll}
		\Rightarrow & 0 & a & b & c & 1 \\
		\hline
		0           & 1 & 1 & 1 & 1 & 1 \\
		a           & b & 1 & b & 1 & 1 \\
		b           & c & c & 1 & c & 1 \\
		c           & b & 1 & b & 1 & 1 \\
		1           & 0 & c & b & c & 1
	\end{array}
	\quad
	\begin{array}{l|llllll}
		\Rightarrow & 0 & a & b & c & d & 1 \\
		\hline
		0           & 1 & 1 & 1 & 1 & 1 & 1 \\
		a           & d & 1 & d & 1 & d & 1 \\
		b           & a & a & 1 & 1 & 1 & 1 \\
		c           & 0 & a & d & 1 & d & 1 \\
		d           & a & a & 1 & 1 & 1 & 1 \\
		1           & 0 & a & d & 1 & d & 1
	\end{array}
	\quad
	\begin{array}{l|lllll}
		\Rightarrow & 0 & a & b & c & 1 \\
		\hline
		0           & 1 & 1 & 1 & 1 & 1 \\
		a           & b & c & b & 1 & 1 \\
		b           & a & a & c & 1 & 1 \\
		c           & 0 & a & b & 1 & 1 \\
		1           & 0 & a & b & 1 & 1
	\end{array}
	\]
\hspace*{34mm}{\rm(a)}\hspace*{38mm}{\rm(b)}\hspace*{37mm}{\rm(c)}
\end{example}
	
\begin{remark}
	If $(L,\vee,\wedge,{}^*,0,1)$ is a pseudocomplemented lattice and $a\in L$ then one can easily check the following:
	\[
	\begin{array}{l|lllll}
		\Rightarrow & 0      & a              & a^*            & a^{**}         & 1 \\ \hline
		0           & 1      & 1              & 1              & 1              & 1 \\
		a           & a^*    & a^*\vee a^{**} & a^*            & a^*\vee a^{**} & 1 \\
		a^*         & a^{**} & a^{**}         & a^*\vee a^{**} & a^{**}         & 1 \\
		a^{**}      & a^*    & a^*\vee a^{**} & a^*            & a^*\vee a^{**} & 1 \\
		1           & 0      & a^{**}         & a^*            & a^{**}         & 1
	\end{array} 
	\]
\end{remark}

The implication $\Rightarrow$ is right monotone and left antitone, see the following results.

\begin{lemma}\label{lem9}
	Let $\mathbf L=(L,\vee,\wedge,{}^*,0,1)$ be a pseudocomplemented lattice and $a,b,c\in L$. Then the following holds:
	\begin{enumerate}[{\rm(i)}]
		\item $a\Rightarrow b\le a\Rightarrow(a\Rightarrow b)$,
		\item if $a\le b$ then $c\Rightarrow a\le c\Rightarrow b$ and $b\Rightarrow c\le a\Rightarrow c$,
		\item if $a\le b$ then $a\Rightarrow b\in D(\mathbf L)$,
		\item $(a\Rightarrow b)\Rightarrow a=a^{**}$.
	\end{enumerate}
\end{lemma}

\begin{proof}
	\
	\begin{enumerate}[(i)]
		\item We have $a\Rightarrow b=a^*\vee b^{**}=a^*\vee(a^*\vee b^{**})\le a^*\vee(a^*\vee b^{**})^{**}=a\Rightarrow(a\Rightarrow b)$.
		\item This follows from Lemma~\ref{lem4}.
		\item This follows from Lemma~\ref{lem4}.
		\item According to Lemma~\ref{lem3} we have
\[
(a\Rightarrow b)\Rightarrow a=(a^*\vee b^{**})^*\vee a^{**}=(a^{**}\wedge b^*)\vee a^{**}=a^{**}.
\]
	\end{enumerate}
\end{proof} 

For $\Rightarrow$ we can easily prove a stronger result than (i) of Lemma~\ref{lem4}, in a certain sense the converse of (iii) of Lemma~\ref{lem1}.

\begin{lemma}\label{lem2}
	Let $(L,\vee,\wedge,{}^*,0,1)$ be a pseudocomplemented lattice satisfying the Stone identity and $a,b\in L$. Then the following holds:
	\begin{enumerate}[{\rm(i)}]
		\item If $a\le b^{**}$ then $a\Rightarrow b=1$,
		\item $a\Rightarrow(b\Rightarrow a)=1$.
	\end{enumerate}
\end{lemma}

\begin{proof}
	\
	\begin{enumerate}[(i)]
		\item If $a\le b^{**}$ then $1=b^*\vee b^{**}\le a^*\vee b^{**}=a\Rightarrow b$.
		
		\item We have $1=a^*\vee b^*\vee a^{**}\le a^*\vee(b^*\vee a^{**})^{**}=a\Rightarrow(b\Rightarrow a)$.
	\end{enumerate}
\end{proof}

Since $b\le b^{**}$ for each $b\in L$, from (i) of Lemma~\ref{lem2} we obtain
\begin{enumerate}[(i')]
	\item If $a\le b$ then $a\Rightarrow b=1$.
\end{enumerate}

Further interesting properties of $\Rightarrow$ can be shown provided the lattice $\mathbf L$ is distributive.

\begin{lemma}\label{lem8}
	Let $(L,\vee,\wedge,{}^*,0,1)$ be a distributive pseudocomplemented lattice and $a,b,c\in L$. Then the following holds:
	\begin{enumerate}[{\rm(i)}]
		\item $a\wedge(a\Rightarrow b)\le b^{**}$,
		\item $a\le b\Rightarrow c$ implies $a\wedge b\le c^{**}$,
		\item if $a\rightarrow b=a\Rightarrow b$ then $a^{**}\wedge b=a^{**}\wedge b^{**}$.
	\end{enumerate}
\end{lemma}

\begin{proof}
	\
	\begin{enumerate}[(i)]
		\item This follows from $a\Rightarrow b=a\rightarrow b^{**}$ and Lemma~\ref{lem1}.
		\item This follows from $b\Rightarrow c=b\rightarrow c^{**}$ and Lemma~\ref{lem1}.
		\item If $a\rightarrow b=a\Rightarrow b$ then
\begin{align*}
	a^{**}\wedge b & =(a^{**}\wedge a^*)\vee(a^{**}\wedge b)=a^{**}\wedge(a^*\vee b)=a^{**}\wedge(a\rightarrow b)=a^{**}\wedge(a\Rightarrow b)= \\
	& =a^{**}\wedge(a^*\vee b^{**})=(a^{**}\wedge a^*)\vee(a^{**}\wedge
	b^{**})=a^{**}\wedge b^{**}.
\end{align*}
	\end{enumerate}
\end{proof}

Using the previous results, we can prove some important properties of implications provided the pseudocomplemented lattice is a Stone lattice.

\begin{theorem}\label{th1}
	Let $(L,\vee,\wedge,{}^*,0,1)$ be a Stone lattice and $a,b,c\in L$. Then the following holds:
	\begin{enumerate}[{\rm(i)}]
		\item $(a\Rightarrow b)\Rightarrow b=a^{**}\vee b^{**}=(b\Rightarrow a)\Rightarrow a$ {\rm(}{\em quasi-commutativity}{\rm)},
		\item $a\wedge b^*\le c^{**}$ if and only if $a\le b^*\Rightarrow c$,
		\item $a\rightarrow b=a\Rightarrow b$ if and only if $a^{**}\wedge b=a^{**}\wedge b^{**}$,
		\item $a\Rightarrow(b\Rightarrow c)=(a\Rightarrow c)\vee(b\Rightarrow c)=b\Rightarrow(a\Rightarrow c)$ {\rm(}{\em exchange property}{\rm)},
		\item $(a\Rightarrow b^*)\vee c^*=a\Rightarrow(b^*\vee c^*)=(a\Rightarrow c^*)\vee b^*$,
		\item $a\Rightarrow b=1$ if and only if $a\le b^{**}$.
	\end{enumerate}
\end{theorem}

\begin{proof}
	\
	\begin{enumerate}[(i)]
		\item According to Lemma~\ref{lem3} we have
		\begin{align*}
		(a\Rightarrow b)\Rightarrow b & =(a^*\vee b^{**})^*\vee b^{**}=(a^{**}\wedge b^*)\vee b^{**}=(a^{**}\vee b^{**})\wedge(b^*\vee b^{**})= \\
		                              & =a^{**}\vee b^{**}=b^{**}\vee a^{**}=(b\Rightarrow a)\Rightarrow a.
		\end{align*}
		\item This follows from Proposition~\ref{prop1}.
		\item If $a^{**}\wedge b=a^{**}\wedge b^{**}$ then
		\begin{align*}
			a\rightarrow b & =a^*\vee b=(a^*\vee a^{**})\wedge(a^*\vee b)=a^*\vee(a^{**}\wedge b)=a^*\vee(a^{**}\wedge b^{**})= \\
			               & =(a^*\vee a^{**})\wedge(a^*\vee b^{**})=a^*\vee
			b^{**}=a\Rightarrow b.
		\end{align*}
		The converse implication follows from Lemma~\ref{lem8}.
		\item We have
\begin{align*}
	a\Rightarrow(b\Rightarrow c) & =a^*\vee(b^*\vee c^{**})^{**}=a^*\vee(b^*\vee c^{**})=(a^*\vee c^{**})\vee(b^*\vee c^{**})= \\
	                             & =(a\Rightarrow c)\vee(b\Rightarrow c)=(b\Rightarrow c)\vee(a\Rightarrow c)=b\Rightarrow(a\Rightarrow c).
\end{align*}
\item We have
\begin{align*}
	(a\Rightarrow b^*)\vee c^* & =(a^*\vee b^*)\vee c^*=a^*\vee(b^*\vee c^*)=a^*\vee(b^*\vee c^*)^{**}=a\Rightarrow(b^*\vee c^*)= \\
	                           & =a\Rightarrow(c^*\vee b^*)=(a\Rightarrow c^*)\vee b^*.
\end{align*}
		\item One implication follows from $a\Rightarrow b=a\rightarrow b^{**}$ and Lemma~\ref{lem1} and the converse implication from Lemma~\ref{lem2}.
	\end{enumerate}
\end{proof}

The question if the implication $\Rightarrow$ is ``better'' than the implication $\rightarrow$ has not a simple answer: Although the implication $\Rightarrow$, contrary to the implication $\rightarrow$, satisfies that $x\le y$ implies $x\Rightarrow y=1$ provided the pseudocomplemented lattice satisfies the Stone identity, the implication $\Rightarrow$ does not satisfy Modus Ponens in its genuine version. On the other hand, $\Rightarrow$ satisfies the contraposition law.

\section{Deductive systems}

The concept of a deductive system was introduced for implication reducts of several propositional calculi based on lattices or semilattices, see e.g.\ \cite{CHK}. For our purposes, we define it in two ways, see Definitions~\ref{def1} and \ref{def2} below.

\begin{definition}\label{def1}
	A {\em deductive system of the first kind} of a pseudocomplemented lattice $(L,\vee,\wedge,{}^*,0,1)$ is a subset $A$ of $L$ satisfying the following conditions:
	\begin{itemize}
		\item $1\in A$,
		\item if $x\in A$, $y\in L$ and $x\rightarrow y\in A$ then $y\in A$.
	\end{itemize}
\end{definition}

Since the intersection of deductive systems of the first kind of $\mathbf L$ is again a deductive system of the first kind of $\mathbf L$, the deductive systems of the first kind of $\mathbf L$ form a complete lattice with respect to inclusion with bottom element $\{1\}$ and top element $L$.

\begin{example}
	The deductive systems of the first kind of the pseudocomplemented lattices from Fig.~1 are as follows: \\
	$\{1\}$, $\{b,1\}$, $\{a,c,1\}$, $\{0,a,b,c,1\}$ for the lattice from Fig.~1 {\rm(a)}, \\
	$\{1\}$, $\{c,1\}$, $\{d,1\}$, $\{a,c,1\}$, $\{b,d,1\}$, $\{b,c,d,1\}$, $\{0,a,b,c,d,1\}$ for the lattice from Fig.~1 {\rm(b)}, \\
	$\{1\}$, $\{a,1\}$, $\{b,1\}$, $\{c,1\}$, $\{a,c,1\}$, $\{b,c,1\}$, $\{0,a,b,c,1\}$ for the lattice from Fig.~1 {\rm(c)}.
\end{example}

\begin{lemma}\label{lem10}
	Let $\mathbf L=(L,\vee,\wedge,{}^*,0,1)$ be a distributive pseudocomplemented lattice and $F$ a filter of $\mathbf L$. Then $F$ is a deductive system of the first kind of $\mathbf L$.	
\end{lemma}

\begin{proof}
	Clearly, $1\in F$. If $a\in F$, $b\in L$ and $a\rightarrow b\in F$ then because of $a\wedge(a\rightarrow b)\le b$ according to Lemma~\ref{lem1} we have $b\in F$.
\end{proof}

The converse of Lemma~\ref{lem10} does not hold since in the Stone lattice from Fig.~1 (b) the set $\{b,d,1\}$ is a deductive system of the first kind, but not a filter.

\begin{lemma}
	Let $\mathbf L=(L,\vee,\wedge,{}^*,0,1)$ be a pseudocomplemented lattice and $\Theta\in\Con\mathbf L$. Then $[1]\Theta$ is a deductive system of the first kind of $\mathbf L$ and $[1]\Theta$ is a sublattice of $(L,\vee,\wedge)$.
\end{lemma}

\begin{proof}
	Clearly, $1\in[1]\Theta$. If $a\in[1]\Theta$, $b\in L$ and $a\rightarrow b\in[1]\Theta$ then
	\[
	b=1'\vee b\mathrel{\Theta}a'\vee b=a\rightarrow b\mathrel{\Theta}1.
	\]
	That $[1]\Theta$ is a subuniverse of $(L,\vee,\wedge)$ is clear.
\end{proof}

\begin{definition}\label{def2}
	A {\em deductive system of the second kind} of a pseudocomplemented lattice $(L,\vee,\wedge,{}^*,0,1)$ is a subset $A$ of $L$ satisfying the following conditions:
	\begin{itemize}
		\item $1\in A$,
		\item if $x\in A$, $y\in L$ and $x\Rightarrow y\in A$ then $y\in A$.
	\end{itemize}
\end{definition}

Since the intersection of deductive systems of the second kind of $\mathbf L$ is again a deductive system of the second kind of $\mathbf L$, the deductive systems of the second kind of $\mathbf L$ form a complete lattice with respect to inclusion with bottom element $\{1\}$ and top element $L$.

\begin{example}
	The deductive systems of the second kind of the pseudocomplemented lattices from Fig.~1 are as follows: \\
	$\{1\}$, $\{b,1\}$, $\{a,c,1\}$, $\{0,a,b,c,1\}$ for the lattice from Fig.~1 {\rm(a)}, \\
	$\{1\}$, $\{c,1\}$, $\{a,c,1\}$, $\{b,c,d,1\}$, $\{0,a,b,c,d,1\}$ for the lattice from Fig.~1 {\rm(b)}, \\
	$\{1\}$, $\{c,1\}$, $\{a,c,1\}$, $\{b,c,1\}$, $\{0,a,b,c,1\}$ for the lattice from Fig.~1 {\rm(c)}.
\end{example}

The relationship between deductive systems of the first and second kind is as follows.

\begin{lemma}\label{lem5}
	Let $\mathbf L=(L,\vee,\wedge,{}^*,0,1)$ be a pseudocomplemented lattice satisfying the Stone identity and $A$ a deductive system of the second kind of $\mathbf L$. Then the following holds:
	\begin{enumerate}[{\rm(i)}]
		\item If $a\in A$, $b\in L$ and $a\le b$ then $b\in A$,
		\item $A$ is a deductive system of the first kind of $\mathbf L$.
	\end{enumerate}
\end{lemma}

\begin{proof}
	\
	\begin{enumerate}[(i)]
		\item If $a\in A$, $b\in L$ and $a\le b$ then $1=b^*\vee b^{**}\le a^*\vee b^{**}=a\Rightarrow b$, i.e.\ $a\Rightarrow b=1\in A$ and hence $b\in A$.
		\item If $a\in A$, $b\in L$ and $a\rightarrow b\in A$ then because of $a\rightarrow b\le a\Rightarrow b$ we have $a\Rightarrow b\in A$ according to (i) and hence $b\in A$.
	\end{enumerate}
\end{proof}

Fig.~1 (c) shows that also in the case when $\mathbf L$ does not satisfy the Stone identity it may happen that every deductive system of the second kind of $\mathbf L$ is a deductive system of the first kind of $\mathbf L$.

For every deductive system $A$ of the second kind of a pseudocomplemented lattice $(L,\vee,\wedge,$ ${}^*,0,1)$ define a binary relation $\Theta(A)$ on $L$ by
\[
(x,y)\in\Theta(A)\text{ if and only if }x\Rightarrow y\in A\text{ and }y\Rightarrow x\in A
\]
($x,y\in L$).

\begin{theorem}
	Let $\mathbf L=(L,\vee,\wedge,{}^*,0,1)$ be a Stone lattice and $A$ a deductive system of the second kind of $\mathbf L$. Then the following holds:
	\begin{enumerate}[{\rm(i)}]
		\item $\Theta(A)$ is reflexive, symmetric and compatible with all fundamental operations of $\mathbf L$ and $[1]\big(\Theta(A)\big)=\{x\in L\mid x^{**}\in A\}$,
		\item if $A$ is closed with respect to $\wedge$ then $\Theta(A)\in\Con\mathbf L$.
	\end{enumerate}
\end{theorem}

\begin{proof}
	Let $a,b,c\in L$.
	\begin{enumerate}[(i)]
		\item Because of $a\Rightarrow a=a^*\vee a^{**}=1\in A$, the relation $\Theta(A)$ is reflexive. Clearly, $\Theta(A)$ is symmetric. Now assume $(a,b)\in\Theta(A)$. Then $a\Rightarrow b\in A$ and $b\Rightarrow a\in A$ and we have
		\begin{align*}
			(a\vee c)\Rightarrow(b\vee c) & =(a\vee c)^*\vee(b\vee c)^{**}=(a^*\wedge c^*)\vee(b^{**}\vee c^{**})= \\
	 		                                  & =(a^*\vee b^{**}\vee c^{**})\wedge(c^*\vee b^{**}\vee c^{**})=a^*\vee b^{**}\vee c^{**}\ge a^*\vee b^{**}= \\
			                                  & =a\Rightarrow b\in A, \\
			(a\wedge c)\Rightarrow(b\wedge c) & =(a\wedge c)^*\vee(b\wedge c)^{**}=(a^*\vee c^*)\vee(b^{**}\wedge c^{**})= \\
			                                  & =(a^*\vee c^*\vee b^{**})\wedge(a^*\vee c^*\vee c^{**})=a^*\vee c^*\vee b^{**}\ge a^*\vee b^{**}= \\
			                                  & =a\Rightarrow b\in A, \\
			               a^*\Rightarrow b^* & =b\Rightarrow a\in A.
		\end{align*}
		According to Lemma~\ref{lem5} (i) we obtain $(a\vee c)\Rightarrow(b\vee c),(a\wedge c)\Rightarrow(b\wedge c),a^*\Rightarrow b^*\in A$. Analogously, $(b\vee c)\Rightarrow(a\vee c),(b\wedge c)\Rightarrow(a\wedge c),b^*\Rightarrow a^*\in A$ can be proved. This shows $(a\vee c,b\vee c),(a\wedge c,b\wedge c),(a^*,b^*)\in\Theta(A)$. Finally, the following are equivalent: $a\in[1]\big(\Theta(A)\big)$; $a\Rightarrow1\in A$ and $1\Rightarrow a\in A$; $a^*\vee1^{**}\in A$ and $1^*\vee a^{**}\in A$; $1,a^{**}\in A$; $a^{**}\in A$.
		\item Assume $(a,b),(b,c)\in\Theta(A)$. Then $a\Rightarrow b\in A$ and $b\Rightarrow c\in A$ and hence $(a\Rightarrow b)\wedge(b\Rightarrow c)\in A$. Now we have
\begin{align*}
	\big((a\Rightarrow b)\wedge(b\Rightarrow c)\big)\Rightarrow(a\Rightarrow c) & =\big((a^*\vee b^{**})\wedge(b^*\vee c^{**})\big)^*\vee(a^*\vee c^{**})^{**}= \\
	& =\big((a^{**}\wedge b^*)\vee(b^{**}\wedge c^*)\big)\vee a^*\vee c^{**}= \\
	& =(a^{**}\vee b^{**}\vee a^*\vee c^{**})\wedge(a^{**}\vee c^*\vee a^*\vee c^{**})\wedge \\
	& \hspace*{5mm}\wedge(b^*\vee b^{**}\vee a^*\vee c^{**})\wedge(b^*\vee c^*\vee a^*\vee c^{**})= \\
	& =1\wedge1\wedge1\wedge1=1\in A.
\end{align*}
Since $A$ is a deductive system of the second kind of $\mathbf L$ we obtain $a\Rightarrow c\in A$. Analogously, $c\Rightarrow a\in A$ can be proved. This shows $(a,c)\in\Theta(A)$ completing the proof of transitivity of $\Theta(A)$.
	\end{enumerate}
\end{proof}

\section{Concluding remark}

The goal of our paper was to introduce a connective implication in pseudocomplemented lattices as a term operation. We succeeded it in two different ways. The question is if such implications satisfy properties required for implication in classical or many-valued logics. It is well-known (see e.g.\ \cite A) that the implication in the so-called implication algebras derived from Boolean algebras are fully determined by the following three axioms:
\begin{enumerate}[(a)]
	\item $(x\rightarrow y)\rightarrow x\approx x$ (contraction property),
	\item $(x\rightarrow y)\rightarrow y\approx(y\rightarrow x)\rightarrow x$ (quasi-commutativity),
	\item $x\rightarrow(y\rightarrow z)\approx y\rightarrow(x\rightarrow z)$ (exchange property),
\end{enumerate}
see e.g.\ \cite A. What we get in our study is the following. If the pseudocomplemented lattice in question is a Stone lattice, then the implication $\Rightarrow$ satisfies quasi-commutativity, the exchange property and the identity
\[
(x\Rightarrow y)\Rightarrow x\approx x^{**},
\]
see e.g.\ Theorem~\ref{th1} (i) and (iv) and Lemma~\ref{lem9} (iv). Thus our implication $\Rightarrow$ differs only in the point (a) where, however, $x^{**}\ge x$.

Moreover, we can compare our implication $\Rightarrow$ with that derived in many-valued Lukasiewicz logic, see e.g.\ \cite{CHK}. Here the implication is fully determined by quasi-commutativity, the exchange property and the identities
\[
x\rightarrow1\approx1\text{ and }1\rightarrow x\approx x.
\]
As mentioned above, $\Rightarrow$ satisfies quasi-commutativity, the exchange property and the identity $x\Rightarrow1\approx1$, thus it differs only in the last axiom where we have the identity
\[
1\Rightarrow x\approx x^{**}
\]
with $x^{**}\ge x$. We hope that this justifies our investigation of implications derived by a pseudocomplemented lattices, in particular of Stone ones.

Authors' addresses:

Ivan Chajda \\
Palack\'y University Olomouc \\
Faculty of Science \\
Department of Algebra and Geometry \\
17.\ listopadu 12 \\
771 46 Olomouc \\
Czech Republic \\
ivan.chajda@upol.cz

Helmut L\"anger \\
TU Wien \\
Faculty of Mathematics and Geoinformation \\
Institute of Discrete Mathematics and Geometry \\
Wiedner Hauptstra\ss e 8--10 \\
1040 Vienna \\
Austria, and \\
Palack\'y University Olomouc \\
Faculty of Science \\
Department of Algebra and Geometry \\
17.\ listopadu 12 \\
771 46 Olomouc \\
Czech Republic \\
helmut.laenger@tuwien.ac.at


\begin{thebibliography}{99}
\bibitem A
J.~C.~Abbott, Semi-boolean algebra. Mat.\ Vesnik {\bf4} (1967), 177--198.
\bibitem{BH}
R.~Balbes and A.~Horn, Stone lattices. Duke Math.\ J.\ {\bf37} (1970), 537--545.
\bibitem{B08}
L.~E.~J.~Brouwer, De onbetrouwbaarheid der logische principes. Tijdschrift Wijsbegeerte {\bf2} (1908), 152--158.
\bibitem{B13}
L.~E.~J.~Brouwer, Intuitionism and formalism. Bull.\ Amer.\ Math.\ Soc.\ {\bf20} (1913), 81--96.
\bibitem{CHK}
I.~Chajda, R.~Hala\v s and J.~K\"uhr, Semilattice Structures. Heldermann, Lemgo 2007. ISBN 978-3-88538-230-0.
\bibitem{CL}
I.~Chajda and H.~L\"anger, Algebraic structures formalizing the logic with unsharp implication and negation. Logic J.\ IGPL (2023) (13 pp.).
\bibitem F
O.~Frink, Pseudo-complements in semi-lattices. Duke Math.\ J.\ {\bf29} (1962), 505--514.
\bibitem G
G.~Gr\"atzer, Lattice Theory: Foundation. Birkh\"auser, Basel 2011. ISBN 978-3-0348-0017-4.
\bibitem{GS}
G.~Gr\"atzer and E.~T.~Schmidt, On a problem of M.~H.~Stone. Acta Math.\ Acad.\ Sci.\ Hungar.\ {\bf8} (1957), 455--460.
\bibitem H
A.~Heyting, Die formalen Regeln der intuitionistischen Logik. I. Sitzungsber.\ Akad.\ Berlin 1930, 42--56.
\bibitem S
T.~P.~Speed, On Stone lattices. J.\ Austral.\ Math.\ Soc.\ {\bf9} (1969), 297--307.
\end{thebibliography}
\end{document}